\documentclass[12pt]{article} 

\usepackage{amsmath}
\usepackage{amsthm}
\usepackage{amsfonts}
\usepackage{mathrsfs}
\usepackage{stmaryrd}
\usepackage{setspace}
\usepackage{fullpage}
\usepackage{amssymb}
\usepackage{breqn}
\usepackage{enumitem}
\usepackage{bbold} 
\usepackage{authblk}
\usepackage{comment}
\usepackage{hyperref}
\usepackage{pgf,tikz}
\usepackage{graphicx}
\usetikzlibrary{decorations.pathreplacing,arrows}
\tikzstyle{vertex}=[circle,draw=black,fill=black,inner sep=0,minimum size=3pt,text=white,font=\footnotesize]

\newtheorem{thm}{Theorem}[section]%[chapter]

\newtheorem{lemma}[thm]{Lemma}

\newtheorem{proposition}[thm]{Proposition}

\newtheorem{clm}[thm]{Claim}

\newtheorem*{lemma*}{Lemma}
\newtheorem*{proposition*}{Proposition}
\newtheorem*{theorem*}{Theorem}

\newcommand\ex{\ensuremath{\mathrm{ex}}}

\newcommand\cH{{\mathcal H}}

\newcommand\cN{{\mathcal N}}

\newcommand{\ignore}[1]{}

\title{Generalized Tur\'an results for matchings}
\author{Dániel Gerbner\\ \small HUN-REN Alfr\'ed R\'enyi Institute of Mathematics\\
\small \texttt{gerbner.daniel@renyi.hu}}
\date{}

\begin{document}

\maketitle

\begin{abstract} 

Given graphs $H$ and $F$, the \textit{generalized Tur\'an number} $\mathrm{ex}(n,H,F)$ is the largest number of copies of $H$ in $n$-vertex $F$-free graphs. We study the case when either $H$ or $F$ is a matching. We obtain several asymptotic and exact results.
\end{abstract}

\section{Introduction}

A fundamental result in extremal graph theory is the theorem of Tur\'an \cite{T}, that determines the largest number of edges in $n$-vertex $K_k$-free graphs. The analogous problems with other forbidden graphs have been widely studied. Recently, a more general problem has attracted a lot of attention under the name \emph{generalized Tur\'an problem}. Given graphs $H$ and $G$, let $\cN(H,G)$ denote the number of copies of $H$ in $G$. We let $\ex(n,H,F)$ denote the largest $\cN(H,G)$ where $G$ is an $n$-vertex $F$-free graph. After several sporadic results, the systematic study of this function has been initiated by Alon and Shikhelman \cite{AS}.

In this paper we study this function in the case $H$ or $F$ is a matching $M_k$, consisting of $k$ independent edges.
Let us highlight some of our results. Gerbner \cite{ger} showed that if $F$ is not a forest, then $\ex(n,M_t,F)=(1+o(1))\ex(n,F)^t/t!$. We extend this to trees, and characterize when forests have this property. We obtain exact results when stars or paths are forbidden. 

In the case of forbidden matchings, the order of magnitude is known, and it is not hard to see that the extremal graph can be chosen at most $c(s)$ ways for some constant $c(s)$ that depends on $s$ but not on $n$. We determine the unique extremal graph in the case the vertex cover number of $H$ (the smallest number of vertices incident to each edge) is at most $s-1$ and $n$ is large enough.

Section 2 deals with the case of counting matchings and Section 3 deals with the case of forbidding matchings.

\section{Counting matchings}

\begin{proposition}\label{aszi}
        For any positive integer $t$ and any tree $F$, we have $\ex(n,M_t,F)=(1+o(1))\ex(n,F)^t/t!$.
\end{proposition}

\begin{proof}
The upper bound is obvious, in an $n$-vertex $F$-free graph $G$ we can count copies of $M_t$ by picking $t$ independent edges, each at most $\ex(n,F)$ ways. We count each copy of $M_t$ exactly $t!$ times.

Let $c:=\limsup \frac{\ex(n,F)}{n}$. Let $\varepsilon>0$. Then there is $m$ such that $\ex(m,F)\ge (c-\varepsilon)m$, let $G_m$ be an $m$-vertex
$F$-free graph with $\ex(m,F)$ edges. For every sufficiently large $n$, we have that $\lfloor n/m\rfloor$ copies of $G_m$ is $F$-free with at least 
$(c-\varepsilon)m\lfloor n/m\rfloor\ge (c-\varepsilon)n-cm\ge (c-2\varepsilon)n$ edges and maximum degree at most $m-1$. Then we claim that $\cN(M_i,G_m)\ge(c-2\varepsilon)^in^i/i!$.
We prove this by induction on $i$. For $M_1$ it is obvious, for $M_{i+1}$ we can count the copies by picking $M_{i}$, and then an edge from another component. This can be done at least $(c-\varepsilon)m(\lfloor n/m\rfloor-i)(c-2\varepsilon)^in^i/i!\ge (c-2i\varepsilon)^in^i(c-2\varepsilon)n$ ways and we count each $M_{i+1}$ exactly $(i+1)$ times. We obtained that for any $\varepsilon >0$, for every sufficiently large $n$, $\ex(n,M_t,F)\ge (c-2\varepsilon)^tn^t/t!$ while $\ex(n,F)\le cn$. This completes the proof.    
\end{proof}

The above argument also works for forests where one of the components $F_1$ has that $\ex(n,F)=(1+o(1))\ex(n,F_1)$. We show that 
the above statement does not hold for other forests. First we need a lemma.

\begin{lemma}\label{lemi} Let $t>1$. For every $\alpha>0$ and $c>0$ there is $\beta>0$ such that the following holds.
   Let $G$ be an $n$-vertex graph with at most $cn$ edges and minimum degree degree at least $\beta n$, then $\cN(M_t,G)\le (1-\alpha)(cn)^t/t!$.
\end{lemma}

\begin{proof}
    We can count the copies of $M_t$ by picking an edge $uv$ and then $t-1$ other edges, independent of $uv$ and from each other. Each copy of $M_t$ is counted $t!$ times this way. The first edge can be picked at least $\beta n$ ways such that $u$ has degree at least $\beta n$. In that case the other edges are not among the edges incident to $u$, thus there are at most $|E(G)|-\beta n$ of them. Therefore, 
    $t!\cN(M_t,G)\le \beta n(|E(G)|-\beta n)^{t-1}+(|E(G)|-\beta n)|E(G)|^{t-1}\le$ $\beta n(|E(G)-\beta n)|E(G)|^{t-2}+(|E(G)|-\beta n)|E(G)|^{t-1}=|E(G)|^t-\beta n^2|E(G)|^{t-2}\le (cn)^t(1-\alpha)$ if $\beta$ is sufficiently small.
\end{proof}

\begin{proposition}
    Let $F$ be a forest with components $F_1,\dots,F_k$ and assume that there exists $\alpha>0$ such that for each $i\le k$, $\ex(n,F_i)<(1-\alpha)\ex(n,F)$. Then there exists $\alpha'>0$ such that $\ex(n,M_t,F)<(1-\alpha')\ex(n,F)^t/t!$.
\end{proposition}

%ja nem lesz jó? $P_2\cup P_100$-ra ha az extremális gráfban van $P_{100}$, akkor arról lógnak le csúcsok... hát, nem látom...

\begin{proof}
    Let $G$ be an $n$-vertex $F$-free graph with $\cN(M_t,G)=\ex(n,M_t,F)$. If $G$ is $F_1$-free, we are done, since $\cN(M_t,G)\le |E(G)|^t/t!\le \ex(n,F_1)^t/t!<(1-\alpha)^t\ex(n,F)^t/t!<(1-\alpha)\ex(n,F)^t/t!$. Consider a copy of $F_1$ in $G$ and let $U_1$ be the set of its vertices. Let $G_1$ be obtained from $G$ by deleting $U_1$. If there is a copy of $F_1$ in $G_1$, let $U_2$ be the set of its vertices and $G_2$ be obtained from $G_1$ by deleting $U_2$. We repeat this, with $U_i$ being the vertex set of a copy of $F_1$ in $G_{i-1}$ (if exists) and $G_i$ obtained from $G_{i-1}$ by deleting $U_{i}$. Let $m=|V(F)|$. 

If $G_{m+1}$ exists, we found $m+1$ vertex-disjoint copies of $F_1$. We apply induction on $k$. Assume first that $k=2$. A copy of $F_2$ in $G$ would have to intersect each copy of $F_1$, which is impossible, thus $G$ is $F_2$-free and we are done. Assume that $k\ge 3$ and let $F'$ be the graph we obtain from $F$ by deleting $F_1$. Then $G_1$ is $F'$-free, thus has at most $(1-\alpha'')\ex(n,F')^t/t!\le (1-\alpha'')\ex(n,F)^t/t!$ edges for some $\alpha''>0$, by induction. This completes the proof.

If $G_{m+1}$ does not exist, let $i$ be largest integer such that $G_i$ exists. Then $G_i$ is $F_1$-free, thus has at most $\ex(n-i|V(F_1)|,F_1)<(1-\alpha)\ex(n,F)$ edges. Let $U=\cup_{i=1}^i U_i$. The copies of $M_t$ not inside $G_i$ each contain an edge incident to at least one of the vertices in $U$. If there are less than $\beta |U| n$ such edges for some small enough $\beta$, then they participate in an additional $\beta |U| n\ex(n,F)^{t-1}<(\alpha-\alpha')\ex(n,F)^t/t!$ copies of $M_t$ for some sufficiently small $\alpha'>0$, completing the proof. If there are at least $\beta |U| n$ such edges, then there is a vertex of degree at least $\beta n$, thus Lemma \ref{lemi} completes the proof.
\end{proof}

Next we show two simple examples, where we determine $\ex(n,M_t,F)$ for every $n$. The \textit{friendship graph} $F_n$ consists of a vertex of degree $n-1$ and a largest possible matching on the other $n-1$ vertices.
%For sake of completeness we deal with every possible value of $n$. ide lehetne minden lin forestre van $\alpha>0$ hogy $\ex(n,M_t,F)<(1-\alpha)(\ex(n,F)^t$, ami Lidicky Liu,Palmer-ben van. Ehhez az elég hogy extrem gráfban van lin fokú pont. Komponens-számra indukció, veszünk egy kompot, maradékra igaz vmi $\alpha_0$-val. Akkor oké ha kevés $M_t$ van amiben itteni pont is van. Márpedig olyanokat összeszámolhatjuk úgy hogy itteni pont, szomszéd, stb, tehát oké

\begin{proposition}
    \textbf{(i)} $\ex(n,M_t,K_2\cup P_3)=\max\{\cN(M_t,K_4),\cN(M_t,M_{\lfloor n/2\rfloor})$.
    
\textbf{(ii)} $\ex(n,M_t,2P_3)=\max\{\cN(M_2,K_5\cup M_{\lfloor (n-5)/2\rfloor}),\cN(M_2,K_4\cup M_{\lfloor (n-4)/2\rfloor}),\cN(M_2,F_n)\}$. 
\end{proposition} 
%30+10floor(n-5)/2+floor (n-5)/2-nek a négyzete vs (n-1)floor n-3)/2+floor (n-1)/2-nek a négyzete. Vagy inkább $M_t$? Akkor kell a $K_4\cup M...$ is

Note that in \textbf{(i)} it is clear that $K_4$ is better than the other construction if and only if $t=2$ and $n=4,5$. In \textbf{(ii)}, the situation is more complicated, but clearly $K_5\cup M_{\lfloor (n-5)/2\rfloor})$ is the best if $n$ is sufficiently large.

\begin{proof}
%Let $G$ be an $n$-vertex $2P_3$-free graph. If there is a 2-edge path $uvw$ in $G$, then every other $P_3$ is incident to at least one of these vertices. Let $E_1$ be the set of edges incident to them and $E_2$ be the set of other edges. Then $E_2$ is a matching, thus $|E_2|\le n/2$. In $E_1$ there are at most $2n$ copies of $M_2$. Indeed, if two vertices of $u,v,w$ are incident to at least five edges, we have $2P_3$. This also shows that $|E_1|\le 3+n-3+1+1$. There are at most $n^2/4$ copies of $M_2$ in $E_2$ and at most $|E_1||E_2|\le n|E_1|/2$ other copies of $M_2$.inkább: 
Let $G$ be an $n$-vertex $K_2\cup P_3$-free graph. Assume that $G$ contains two independent edges $uv$ and $xy$. Then each $P_3$ is inside these four vertices. If there is no $P_3$ in $G$, we are done, thus let us assume that $ux\in E(G)$. Then there are no edges outside these vertices, hence we are done with the proof of \textbf{(i)}.

Let $G$ be an $n$-vertex $2P_3$-free graph and $v$ be a vertex of largest degree in $G$. If $d(v)\ge 5$, then there is no $P_3$ that does not contain $v$, thus edges not incident to $v$ form a matching and $G$ is a subgraph of the friendship graph. If $d(v)= 4$, then we can have copies of $P_3$ among the neighbors of $v$. Then the graph is a subgraph of $K_5$ plus a matching. If $d(v)=3$, then we can have copies of $P_3$ that avoid $v$ but contain two of the neighbors of $v$. Each of the three neighbors $x,y,z$ of $v$ can have at most one neighbor in $V(G)\setminus \{v,x,y,z\}$. Let $U$ denote the vertices that are of distance at most 2 from $v$, then $|U|\le 7$. The rest of the graph $G[V(G)\setminus U]$ forms a matching. If $|U|\le 5$, then we can replace $U$ by a clique to obtain a subgraph of the claimed extremal graphs. If $|U|=6$, then, say, $x$ has a neighbor $x'$ and $y$ has a neighbor $y'$ in addition to $v,x,y,z$. Furthermore, $x'$ may be adjacent to $z$. In any case, $z$ cannot be adjacent to $x$ or $y$, as for example $xz$ with $xx'$ would create a $P_3$ disjoint from $vyy'$. Similarly, we cannot have both $x'z$ and $xy$ in the graph, as then $y'yx$ and $x'zv$ would form $2P_3$. 
It is easy to check that there are at most 6 edges, at most 7 copies of $M_2$ and at most 2 copies of $M_3$ inside $U$. If we replace it by $K_4\cup K_2$, each of these values increase. As each copy of $M_t$ intersects $U$ in 0,1,2 or 3 independent edges, the number of copies of $M_t$ increases and we are done.

 Finally, if $|U|=7$, then similarly to the previous case, we have edges $xx'$, $yy'$, $zz'$ and we cannot have edges between $x,y,z$. Then we have 6 edges, 9 copies of $M_2$ and 3 copies of $M_3$ inside $U$, it is better to replace them by $K_5\cup M_1$.

 If $d(v)\le 2$, then each component is a path or a cycle, and there is at most one component that is not a single edge. Moreover, that component has at most 5 vertices, thus a subgraph of $K_5$ and $G$ is a subgraph of $K_5\cup M_{\lfloor (n-5)/2\rfloor})$, completing the proof.
\end{proof}

In light of Proposition \ref{aszi}, we are interested in exact results. In the case the forbidden graph has chromatic number at least 3, the following notion helps. We say that a graph $H$ is \textit{$F$-Tur\'an-stable} if any $n$-vertex $F$-free graph $G$ with $\ex(n,H,F)-o(n^{|V(H)|})$ copies of $H$ can be turned into $T(n,\chi(F)-1)$ by adding and removing $o(n^2)$ edges. Here $T(n,k)$ is the \textit{Tur\'an graph}, a complete $k$-partite graph with each part of order $\lfloor n/k\rfloor$ or $\lceil n/k\rceil$. The Erd\H os-Simonovits stability theorem \cite{erd1,erd2,simi} states that $M_1$ is $F$-Tur\'an-stable for every $F$ with chromatic number at least 3. Gerbner \cite{ger2} observed that the vertex-disjoint union of $F$-Tur\'an-stable graphs is also $F$-Tur\'an-stable, thus $M_t$ has this property for every $F$ with chromatic number at least 3. Gerbner \cite{ger2} also showed that this implies that $\ex(n,M_t,F)=\cH(M_t,T(n,\chi(F)-1))$ for every $F$ with a color-critical edge. Some other exact results that we do not state here are implied by \cite{ger3}.

In the case the forbidden graph is bipartite, we do not have many exact results even if $t=1$. Let us deal with some forbidden graphs where the case $t=1$ is known.
We say that a graph $G$ is \textit{almost $d$-regular} if either every vertex has degree $d$ or all but one of the vertices have degree $d$ and the remaining vertex has degree $d-1$. It is easy to see that $\ex(n,S_r)=|E(G)|$ for some almost $(r-1)$-regular $n$-vertex graph $G$.

\begin{thm}
    For any $t,r$ and sufficiently large $n$, we have $\ex(n,M_t,S_r)=\cN(M_t,G)$ for some almost $(r-1)$-regular $n$-vertex graph $G$. Moreover, if an $n$-vertex $S_r$-free graph $G'$ is not almost $(r-1)$-regular, then $\cN(M_t,G')=\ex(n,M_t,S_r)-\Omega(n^{t-1})$.
\end{thm} 

Note that this theorem does not completely determine $\ex(n,M_t,S_r)$, since  almost $(r-1)$-regular $n$-vertex graphs may contain different number of copies of $M_t$. For example, consider $C_n$ and $n/3$ vertex-disjoint copies of $K_3$. The first edge can be chosen $n$ ways, and the second edge can be chosen $n-3$ ways. However, in the union of triangles the third edge can be chosen $n-6$ ways always. In $C_n$ for some choices of the first two edges, the third edge can be chosen $n-5$ ways, while for each other choice of the first two edges the third edge can be chosen $n-6$ ways, thus $C_n$ contains more cpies of $M_3$ than $n/3$ vertex-disjoint copies of $K_3$.

\begin{proof}
We use induction on $t$, the statement is trivial for $t=1$. Consider an $S_r$-free graph $G'$. We count the copies of $M_t$ by picking $M_{t-1}$ first, and then an independent edge. Observe that we can pick $M_{t-1}$ $\cN(M_{t-1},G')$ ways, and then an edge can be picked among those that are not incident to the $2t-2$ vertices in the matching we picked. This can be done at least $|E(G')|-(2t-2)(r-1)$ ways.

In an almost $(r-1)$-regular graph $G$, when we pick $i$ independent edges, all but $O(1)$ other edges each have that both endpoints are not adjacent the $2i$ vertices picked earlier. We can pick the edges of an $M_i$ one by one satisfying the above property. This way we obtain all but $O(n^{i-1})$ copies of $M_i$. The above property implies that when we pick the next edge, we always have exactly $|E(G)-2i(r-2)$ or $|E(G)-2i(r-2)+1$ choices, the second possibility is when one of the $2i$ vertices have degree $r-2$ (and this happens $O(n^{i-1})$ times). Indeed, we have to avoid the edges that are incident to the already picked vertices, and we ensured that no edge is counted twice. Note that out of these $|EG)|-O(1)$ edges, $O(1)$ edges have a neighbor adjacent to the $2i$ vertices picked earlier. 

%In $G$, the second edge can be picked at least $|E(G)|-2r+3$ ways. It is important to note that $|E(G')|<|E(G)|$.

Assume that the statement holds for $t-1$ and consider $|E(G')|$. If $\cN(M_{t-1},G)=o(n^{t-1})$, then clearly $\cN(M_t,G)=o(n^t)$ and we are done, hence we assume that $\cN(M_{t-1},G)=\Theta(n^{t-1})$. If $|E(G')|\le (r-1)n/2-2tr$, then we can pick the last edge at most $|E(G)|-2tr$ ways, thus each term is less than doing the same in $G$, and there are less terms, but still $\Theta(n^{t-1})$ many, so we lose $\Omega(n^{t-1})$. If $|E(G')|>(r-1)n/2-2tr$, then all but at most $4tr$ of the vertices have degree $r-1$, thus all but a set $E_0$ of at most $4tr^2$ edges have endpoints of degree $r-1$. For the copies of $M_{t-1}$ avoiding $E_0$, the last edge can be picked at most $|E(G')|-2(t-1)(r-2)<|E(G)|-2(t-1)(r-2)$ ways, thus we lose at least one copy of $M_t$ for each such copy of $M_{t-1}$, and each copy of $M_t$ is counted at most $t$ times. The $O(n^{t-2})$ other copies of $M_{t-1}$ are each contained in at most $|E(G')|<|E(G)|-1$ copies of $M_t$. Compared to $G$, we have $\Theta(n^{t-1})$ copies of $M_{t-1}$ that are contained in less copies of $M_t$ and $O(n^{t-2})$ copies that may be contained in more copies of $M_t$ by $O(1)$. This implies that $\cN(M_t,G')\le \cN(M_t,G)-\Theta(n^{t-1})+O(n^{t-2})$,
completing the proof.
\end{proof} 

Let us turn to forbidden paths and let $P_k$ denote the path on $k$ vertices. Faudree and Schelp \cite{FSch} determined $\ex(n,P_k)$, improving a result of
Erd\H os and Gallai \cite{Er-Ga}.
Let $G_{n,k,\ell}=\ell K_{k-1}\cup K_{(k-2)/2}+\overline{K_{n-\ell(k-1)-(k-2)/2}}$.

\begin{thm}[Faudree and Schelp \cite{FSch}]\label{fash}
    We have $\ex(n,P_k)=|E(aK_{k-1}\cup K_b)$ for $n=a(k-1)+b$, where $a$ and $b$ are non-negative integers and $b<k$. Furthermore we have equality if and only if $G=aK_{k-1}\cup K_b$ or $k$ is even, $b=k/2$ or $b=1+k/2$, in which case $G_{n,k,\ell}$ is another extremal graph for $0\le\ell\le a$.
\end{thm}

\begin{thm}\label{pati}
    We have $\ex(n,M_t,P_{k})=\cN(M_t,aK_{k-1}\cup K_b)$ for $n$ sufficiently large and of the form $n=a(k-1)+b$, where $a$ and $b$ are non-negative integers and $b<k$.
\end{thm}

We will need a stability theorem of F\"uredi, Kostochka and Verstra\"ete \cite{fkv}.

\begin{thm}\label{fkvstabi}
    Let $\ell\ge 2$ and $n\ge 3\ell-1$ and $k\in 2\ell+1,2\ell+2$, and let $G$ be a connected $n$-vertex graph containing no $k$-vertex path. Then $|E(G)|\le \binom{k-1}{2}+(\ell-1)(n-k+\ell+1)$ unless (a) $k=2\ell$, $k\neq 6$ and $G\subseteq H(n,k,\ell)$ or (b) $k=2\ell+1$ or $k=6$ and we can obtain a star forest from $G$ by deleting a set $A$ of at most $\ell-1$ vertices.
\end{thm}

Here $H(n,k,a)$ is the following graph. We take a set $A$ of order $a$, a set $B$ of order $n-k+a$ and a set $C$ of order $k-2a$. We add all the edges between $A$ and $B$ and all the edges inside $A\cup C$. In particular, $H(n,k,t-1)$ contains a vertex of degree $n-1$, and this is the only property of this graph we will use in addition to Theorem \ref{fkvstabi}.
%we can apply Lemma \ref{lem1} to show that $H(n,k,t-1)$contains less than $\ex(n,M_t,P_k)$ copies o

\begin{proof}[Proof of Theorem \ref{pati}]
    Let $G$ be an $n$-vertex $P_k$-free graph. Assume first that each component is of order 
    %at most $3\ell-2$. 
    $o(n)$. Assume that there is a component that is $G_{m,k,0}$ for some $m$. We write $m=a'(k-1)+b'$ with $b'<k$. Then we replace this component by $a'K_{k-1}\cup K_{b'}$. By Theorem \ref{fash} the number of edges does not change. We claim that the number of copies of other matchings does not decrease either. Indeed, after picking $i$ independent edges from $G_{m,k,0}$, the next can be picked $|E(G_{m-2i,k-2i,0})|$ ways. Let us write $m-2i=a''(k-2i-1)+b''$ with $b''<k$, then $|E(G_{m-2i,k-2i,0})|=|E(a''K_{k-2i-1}\cup K_{b''})|$. After picking $i$ independent edges from $a'K_{k-1}\cup K_{b'}$, we are left with $a'$ cliques each of order at least $k-2i$, and one clique of order at most $b'$ and at least $k-2i$. We can obtain this graph from $a''K_{k-2i-1}\cup K_{b''}$ by moving vertices from cliques to smaller cliques. Each such step increases the number of edges. This shows that inside the part we changed, the number of copies of $M_i$ does not decrease. As each copy of $M_t$ intersects this part in a copy of $M_i$ for some $i$ and is unchanged outside, we have that the number of copies of $M_t$ does not decrease.

        Assume now that there is a component $K$ on $m$ vertices that is not a clique and not $G_{m,k,0}$. Then we can replace this component by $a'K_{k-1}\cup K_{b'}$. The number of copies of $M_i$ may decrease for $i>1$, but by Theorem \ref{fash}, the number of edges increases. Each copy of $M_t$ intersects $K$ in a copy of $M_i$ for some $i$, and each copy of $M_i$ in $K$ is extended to a copy of $M_t$ at most $O(n^{t-i})$ ways. Therefore, by this change we lose $o(n^{t-1})$ copies of $M_t$ (recall that $m=o(n)$). If there are $\Omega(n^{t-1})$ copies of $M_{t-1}$ outside $K$, then the number of copies of $M_t$ in $G$ increases by $\Omega(n^{t-1})$ because of the increase of the number of edges in $G$. If there are $o(n^{t-1})$ copies of $M_{t-1}$ outside $K$, then there are $o(n^t)$ copies of $M_t$ in $G$ and we are done.

Assume now that there is a component $K'$ of order 
%at least $3\ell-1$.
$m=\Omega(n)$. Then we can apply Theorem \ref{fkvstabi} to this component. We obtain that there are three possibilities. If $K'$ is a subgraph of $H(m,k,\ell)$ or $|E(K')|\le |E(H(m,k,\ell-1))|$, then $\cN(M_i,K)\le \cN(M_i,H(m,k,\ell))<\cN(M_i,a'K_{k-1}+K_{b'})$ for every $1<i\le t$. Indeed, the number of edges in $H(m,k,\ell)$ is not more than in $a'K_{k-1}+K_{b'}$, and there is a vertex of linear degree (except for the trivial case $\ell-1=0$), thus Lemma \ref{lemi} implies the statement about $M_i$. Clearly the number of copies of $M_0$ and $M_1$ does not decrease either. As each copy of $M_t$ intersects $K'$ in an $M_i$ for some $i$, the proof is complete in this case.

If we can obtain a star forest from $K'$ by deleting a set $S$ of at most $\ell-1$ vertices, we count the copies of $M_i$ by picking edges incident to $S$, or edges not incident to $S$. We fix $\alpha>0$. If at any time we pick an edge less than $(1-\alpha)(k-2)n/2$ ways, then we are done. In particular, this is the case if we pick an edge incident to $S$ if there are $o(n)$ edges incident to $S$. If there are $\Omega(n)$ edges incident to $S$, then Lemma \ref{lemi} completes the proof. We are left with the case we pick each edge from the star forest. Clearly there are at most $n-1$ edges there, thus we are done if $k>4$. If $k=4$, $S$ must be empty and it is easy to see that any star forest contains less copies of $M_t$ than $a'K_3+K_{b'}$.
\end{proof}

\section{Forbidden matchings}
We will use the following theorem of Berge and Tutte \cite{berg}.

\begin{thm}[Berge-Tutte]  A graph $G$ is $M_{s}$-free if and only if there is a set $B\subset V(G)$ such that removing $A$ cuts $G$ to connected components $G_1,\dots,G_m$ with each $V(G_i)$ of odd order such that $|B|+\sum_{i=1}^m \frac{|V(G_i)|-1}{2}\le s-1$.
\end{thm}

We call the partitions satisfying the above theorems \textit{Berge-Tutte partitions}. Clearly, we can add any edges incident to some vertex of $A$ and any edges inside a component $G_i$ without increasing $|A|+\sum_{i=1}^m \frac{|V(G_i)|-1}{2}$, and without decreasing the number of copies of $H$. Therefore, there is an $M_s$-free $n$-vertex graph with $\ex(n,H,M_s)$ copies of $H$ where the vertices in $A$ have degree $n-1$ and each $G_i$ is a clique. Given an integer $s$, there are $O_s(1)$ such graphs, hence we can consider this problem mostly solved. Still, given $H$, it is of interest to narrow down this list to those graphs that actually appear as extremal graphs. 
%Wang \cite{wang} determined $\ex(n,K_k,M_s)$ and $\ex(n,H(k,2r,r),M_s)$. Lu, Yuan and Zhang \cite{lyz} determined $\ex(n,K_{a,b},M_s)$. Liu and Zhang \cite{liuz} extended these results and determined $\ex(n,H,M_s)$ for every complete multipartite graph $H$.
 Liu and Zhang \cite{liuz}, extending earlier results \cite{wang,lyz} proved that for every complete multipartite graph $H$ and every $n\ge 2s-1$, $\ex(n,H,M_s)=\max\{\cN(H,K_{2s-1}, \cN(H,H(n,2s-1,s-1))\}$.

 Gerbner \cite{ger4} determined the order of magnitude of $\ex(n,H,sK_k)$ for every $H$. In particular, for matchings it states the following. Let $U\subset V(H)$. A \textit{partial $(m,U)$-blowup} of $H$ is obtained by replacing each vertex $u\in U$ with $m$ vertices $u_1,\dots,u_m$, and replacing each edge by a complete bipartite graph between the corresponding vertices. Let us consider a largest set $U\subset V(H)$ such that no partial $(m, U)$-blowup of $H$ contains $M_s$, and let $b(H,s)$ denote the order of $U$. Then $\ex(n,H,M_s)=\Theta(n^{b(H,s)})$. 
 
 Observe that another characterization can be obtained using the Berge-Tutte theorem. If $H$ contains $M_s$, then clearly $\ex(n,H,M_s)=0$. Otherwise, there is a Berge-Tutte partition of $H$. Let $b$ the largest integer such that there is a Berge-Tutte-partition of $H$ where $b$ vertices belong to one-vertex components $G_i$. Then $\ex(n,H,M_s)=\Theta(n^{b})$, i.e., $b=b(H,s)$. Indeed, $\ex(n,H,M_s)=\Omega(n^{b})$, since we can pick the $b$ vertices from $\Theta(n)$ vertices $\Theta(n^b)$ ways, and $\ex(n,H,M_s)=O(n^b)$, since if we fix a Berge-Tutte partition of an $n$-vertex $M_s$-free graph $G$, there are $O(1)$ ways to embed at least $|V(H)|-b$ vertices to components $G_i$ of order more than 1 and to $A$.

It is not clear how to obtain a simpler characterization of $b(H,s)$. Clearly we cannot blow up the adjacent vertices without creating $M_s$, thus $b(H,s)\le\alpha(H)$. 
Let $\tau(H)$ denote the smallest number of vertices such that each edge is incident to at least one of them. Then the other vertices form a largest independent set, i.e., $\tau(H)=|V(H)|-\alpha(H)$. 
If $\tau(H)\le s-1$, then we can blow up each vertex in an independent set, since that does not increases $\tau(H)$ and $\tau(M_s)=s$, hence in this case $b(H,s)=\alpha(H)$. This can also be easily seen by the construction $H(n,2s-2,s-1)=H(n,2s-1,s-1)$, which is obtained from $K_{s-1,n-s+1}$ by adding all the possible edges inside the part of order $s-1$. We can show that this construction does not only give the order of magnitude, but also an exact result if $n$ is sufficiently large.

\begin{proposition}\label{taus}
    Let $\tau(H)\le s-1$. Then $\ex(n,H,M_s)=\cN(H,H(n,2s-1,s-1)$ if $n$ is sufficiently large.
\end{proposition}

\begin{proof}
    Let $G$ be an $n$-vertex $M_s$-free graph that contains $\ex(n,H,M_s)$ copies of $H$ and consider a Berge-Tutte partition of $G$. Count first the number of copies of $H$ that contain an edge in some $G_i$. Observe that there are $O(1)$ such edges. Then there are $O(1)$ ways to pick some vertices from $A$ and some other vertices that are incident to edges inside some $G_j$. Finally, there are $O(n)$ ways to pick each other vertex from the components $G_j$. Observe that these are independent from each other and from any endpoint of the edge picked first, thus there are at most $\alpha(H)-1$ of them, hence there are $O(n^{\alpha(H)-1})$ such copies of $H$.

    If $|A|<s-1$, we delete each edge inside the sets $G_i$, and join $s-|A|-1$ vertices of $G_i$ to each other vertex. Let $A'$ be the union of $A$ and these $s-|A|-1$ vertices. This way we obtain $H(n,2s-2,s-1)$. We deleted $O(n^{\alpha(H)-1})$ copies of $H$. For each vertex $v$ of $A'\setminus A$, we can find copies of $H$ by picking $v$, $\tau(H)-1$ other vertices of $A'$ and $\alpha(H)$ vertices outside $A'$. This can be done $O(n^{\alpha(H)})$ ways, thus the number of copies of $H$ increases, a contradiction.

    Finally, if the smallest possible $A$ that can be chosen in a Berge-Tutte partition has order $s-1$, then there cannot be an edge inside any $G_i$. Indeed, we prove this by induction on $s$. For $s=2$, $G$ must be a star (plus some isolated vertices) and we can pick the center as $A$. For larger $s$, we use that $n$ is large enough. There is a vertex $u\in A$ that is incident to at least $2s-1$ vertices, as otherwise there are $O(1)$ non-isolated vertices in $G$, thus $G$ contains $O(1)$ copies of $H$. Let us delete $u$ from $G$. The resulting graph $G'$ does not contain $M_{s-1}$, as that could be extended to an $M_s$ in $G$ by the edge $uv$ for a vertex $v$ that is a neighbor of $u$ and is not in the copy of $M_{s-1}$. 
    
    Then by induction, either there is a Berge-Tutte partition of $G'$ with $A_0$ having order less than $s-2$, or a Berge-Tutte partition of $G'$ with $A_0$ having order $s-2$ and no edge in any of the components created by removing $A_0$, where $A_0$ denotes the set of vertices that cuts $G'$ to odd components in the definition of the Berge-Tutte partition. 
    
    In both cases we add $u$ to $A_0$. This way we obtain a Berge-Tutte partition of $G$, where the set $A''=A_0\cup \{u\}$ corresponds to $A$. We have that either $|A''|\le s-2$, or there are no edges outside $A''$, completing the proof that each $G_i$ is a singleton. Then $G$ is a subgraph of $H(n,2s-1,s-1)$, completing the proof.
    %the only other possibility is that the vertex $u\in B$ is adjacent to only one $G_i$, and each edge inside $G_i$ is incident to the neighbor of $u$ inside $G_i$
    %assume otherwise. Then a vertex $u\in B$ is adjacent to $v\in V(G_i)$. Then each edge of 
    %if the largest matching outside $B$ has order $s'$, then let $G'$ be the graph we obtain by deleting the components $G_i$ with edges inside. ...
\end{proof}

We also prove a simple statement in the case $b(H,s)=0$, i.e., $\ex(n,H,M_s)=O(1)$. 

\begin{proposition}
   \textbf{(i)} If every vertex of $H$ has degree at least $s$, then $\ex(n,H,M_s)=\cN(H,K_{2s-1})$ for $n\ge 2s-1$.

      \textbf{(ii)} If every vertex of $H$ has degree at least $s$, except one vertex which has degree $s-1$, then $\ex(n,H,M_s)=\max\{\cN(H(n,2s-1,s-1)),\cN(H,K_{2s-1})\}$ for $n\ge 2s-1$.
\end{proposition}

\begin{proof} Let us start by proving \textbf{(i)}.
    Observe that $b(H,s)=0$ since blowing up any vertex would result in an $M_s$ (and also because any vertex in a one-vertex $G_i$ has degree at most $s-1$). Let $G$ be an $n$-vertex $M_s$-free graph and consider a Berge-Tutte partition of $G$. We claim that there is at most one $G_i$ with vertices contained in some copy of $H$. Indeed, if $|V(G_1)|=a\le |V(G_2)|$, then vertices in $G_1$ have degree at most $a-1+|A|$. Therefore, $a-1+|A|\ge s$. On the other hand, by the definition of the Berge-Tutte partition, $|A|+a-1\le |A|+(|V(G_1)|-1)/2+(|V(G_2)|-1)/2\le s-1$, a contradiction. We obtained that each copy of $H$ is in $A\cup V(G_i)$ for some $i$, which has at most $2s-1$ vertices, completing the proof.

    The proof of \textbf{(ii)} goes similarly. Now it is possible that $G_1$ and $G_2$ both have vertices that are contained in some copies of $H$. However, then each vertex of $G_1$ has degree less than $s$, thus there is exactly one vertex in $G_1$. We claim that there is also exactly one vertex in $G_2$. Indeed, we have $|A|+(|V(G_2)|-1)/2\le s-1$ by the definition of the Berge-Tutte partition, and $|A|\ge s-1$ since the degree of the vertex in $G_1$ is at most $|A|$. But this means we can delete the edges inside components $G_i$ of order more than 1 and then the resulting graph is a subgraph of $H(n,2s-1,s-1)$, completing the proof.
\end{proof}

%utakra: ha van nagy of comp, arra stabi, Furedi, Kostochka... Ha nem, akkor kis kompba kevesebb el, hiaba van tobb $M_t$, osszesen jok leszunk mert sok matching csak 1-et vesz mindenhonnan??? olyanokbol van $n^t$, masbol meg vaszunk egy elt, az o reszebol $O(1)$-felekepp masikat, stb. Szoval van nemkonst komp, ja sot lin meretu komp, azon meg bukunk lin sok elt, az okes, vagy abban a kompban van lin fok, az is oke, vagy keves csucsot leszamitva star forest, van gyok n-es fok. Ha van nagy fok, akkor jok vagyunk? lin fokra prop1.2-ben, konst fokra a csillagosban. A konstrukcioban minden elt legalabb $\ex(n,P_k)-t\binom{k-1}{2}$-felekepp valaszthatjuk, tehat jok vagyunk ha elszmon bukunk konstanst. Ja star forestes esetben ugy kapjuk az $M_t$-t hogy pickelunk eleket vagy a t-1 csucshoz incidensbol, vagy a star forestekbol. Ha az elso lin, akkor van lin foku csucs, ha a masodik, akkor is, mert a $t-1$ osszekoti oket. Lin csucs mar eleg prop1.2-bol ki kell szedni a lemmmat.

\vskip 0.3truecm

\textbf{Funding}: Research supported by the National Research, Development and Innovation Office - NKFIH under the grants FK 132060 and KKP-133819.

\vskip 0.3truecm

%The authors declare that they have no known competing financial interests or personal relationships that could have appeared to influence the work reported in this paper.

\end{document}